\documentclass[11pt]{amsart}

\usepackage[T1]{fontenc}
\usepackage[english]{babel}
\usepackage{csquotes}

\usepackage[a4paper,margin=3cm]{geometry}
\allowdisplaybreaks
\usepackage{graphicx}
\usepackage[dvipsnames]{xcolor}
\usepackage{caption}
\usepackage{subcaption}
\usepackage{tikz}
\usetikzlibrary{
	babel,
	positioning,
	arrows.meta
}
\usepackage{tikz-cd}
\usepackage{pgfplots}
\pgfplotsset{compat=1.18}

\usepackage{mathtools}
\usepackage{amssymb}
\usepackage{esint}
\usepackage{mathrsfs}

\usepackage[shortlabels]{enumitem}
\usepackage{array}
\usepackage{hhline}
\usepackage{comment}
\usepackage[toc,page]{appendix}
\usepackage{imakeidx}
\usepackage{fancyhdr}
\usepackage{ifthen}
\usepackage{forloop}
\usepackage{xstring}
\usepackage{emptypage}
\usepackage{listings}

\usepackage[
backend=biber,
style=numeric,
sorting=nty,
giveninits=true,
uniquename=init,
maxnames=10,
doi=false,
isbn=false,
url=false,
language=auto,
autolang=other
]{biblatex}

\usepackage[
hyperfootnotes=false
]{hyperref}

\hypersetup{
	colorlinks=true,
	linkcolor=black,
	citecolor=black,
	urlcolor=red
}

\usepackage{aliascnt}
\usepackage[
nameinlink,
capitalise,
sort
]{cleveref}

\crefname{equation}{}{}
\Crefname{equation}{}{}
\creflabelformat{equation}{#2(#1)#3}
\crefname{enumi}{}{}
\Crefname{enumi}{}{}

\newcommand*{\newsharedthm}[2]{%
	\newaliascnt{#1}{lemma}%
	\newtheorem{#1}[#1]{#2}%
	\aliascntresetthe{#1}%
}

\theoremstyle{plain}
\newtheorem{lemma}{Lemma}[section]
\newsharedthm{theorem}{Theorem}
\newsharedthm{proposition}{Proposition}
\newsharedthm{corollary}{Corollary}

\theoremstyle{definition}
\newsharedthm{definition}{Definition}
\newsharedthm{conjecture}{Conjecture}
\newsharedthm{claim}{Claim}
\newsharedthm{assumption}{Assumption}

\theoremstyle{remark}
\newsharedthm{remark}{Remark}
\newsharedthm{example}{Example}
\newsharedthm{notation}{Notation}

\newcommand*{\setcrefname}[3]{%
	\crefname{#1}{#2}{#3}%
	\Crefname{#1}{#2}{#3}%
}

\setcrefname{lemma}{Lemma}{Lemmas}
\setcrefname{theorem}{Theorem}{Theorems}
\setcrefname{proposition}{Proposition}{Propositions}
\setcrefname{corollary}{Corollary}{Corollaries}
\setcrefname{definition}{Definition}{Definitions}
\setcrefname{conjecture}{Conjecture}{Conjectures}
\setcrefname{claim}{Claim}{Claims}
\setcrefname{assumption}{Assumption}{Assumptions}
\setcrefname{remark}{Remark}{Remarks}
\setcrefname{example}{Example}{Examples}
\setcrefname{notation}{Notation}{Notations}

\newlist{thmenum}{enumerate}{1}
\setlist[thmenum]{
	label=(\roman*),
	ref=\thetheorem(\roman*)
}
\setcrefname{thmenumi}{Theorem}{Theorems}

\numberwithin{equation}{section}

\newcommand{\R}{\mathbb R}
\newcommand{\Id}{{\rm Id}}
\DeclareMathOperator{\supp}{supp}

\newcommand{\N}{\mathbb{N}}

\begin{document}
	
	\title{Superharmonicity of the fractional ground state}
	
	\begin{abstract}
		We prove that the positive ground state $u$ of the fractional Laplacian $(-\Delta)^s$, $s\in (0,1)$, on an arbitrary open set $\Omega$ satisfies $-\Delta u>\lambda_s(\Omega)^{1/s}u$ in $\Omega$. In one dimension, this yields strong concavity and settles a conjecture of Ba\~{n}uelos, Kulczycki, and M\'{e}ndez-Hern\'{a}ndez. We establish a hierarchy of pointwise inequalities comparing different powers of the Laplacian. In balls, these inequalities give a quantitative Hessian estimate for $\log u$, and hence log-concavity in every dimension. By contrast, we show that log-concavity fails in general convex domains by constructing counterexamples in sufficiently thin ellipsoids. We also extend the superharmonicity principle to ground states of $\psi(-\Delta)$ for complete Bernstein functions $\psi$.
	\end{abstract}
	
	
	
	\author[N.~De Nitti]{Nicola De Nitti}
	\address[N.~De Nitti]{Politecnico di Bari, Dipartimento di Meccanica, Matematica e Management, Via E.~Orabona 4, 70125 Bari, Italy.}
	\email[]{nicola.denitti@poliba.it}
	
	\author[X.~Fern\'andez-Real]{Xavier Fern\'andez-Real}
	\address[X.~Fern\'andez-Real]{EPFL, Institut de Math\'ematiques, Station 8, 1015 Lausanne, Switzerland.}
	\email[]{xavier.fernandez-real@epfl.ch}
	
	\author[F.~Grube]{Florian Grube}
	\address[F.~Grube]{Fakult\"at f\"ur Mathematik, Universit\"at Bielefeld, 33615 Bielefeld, Germany.}
	\email[]{fgrube@math.uni-bielefeld.de}
	
	\author[H.~Prasad]{Harsh Prasad}
	\address[H.~Prasad]{Fakult\"at f\"ur Mathematik, Universit\"at Bielefeld, 33615 Bielefeld, Germany.}
	\email[]{hprasad@math.uni-bielefeld.de}
	
	\maketitle
	%\setcounter{tocdepth}{1}
	%\tableofcontents
	
	\section{Introduction and main results}\label{sec:introduction}

    We study the concavity and superharmonicity of ground states of the fractional Laplacian on an open set $\Omega\subset\R^n$. These satisfy the exterior Dirichlet problem
    \begin{equation}\label{eq:eigenproblem}
	\left\{
	\begin{array}{rcll}
		(-\Delta)^s u&=&\lambda u&\text{in }\Omega,\\
		u&=&0&\text{in }\R^n\setminus\Omega.
	\end{array}
	\right.
	\end{equation}
    The fractional Laplacian is defined on $u\in \mathcal{S}(\R^n)$, $s\in (0,1)$, $n\in \N$ via the singular integral
    \[
	   (-\Delta)^s u(x)\coloneqq c_{n,s}\,\mathrm{P.V.} \int_{\R^n}\frac{u(x)-u(x+y)}{|y|^{n+2s}}\,\mathrm dy, \quad\text{where}\quad c_{n,s}\coloneqq 2^{2s}s\, \frac{\Gamma\!\left(\frac{n+2s}{2}\right)}{\Gamma(1-s)}\pi^{-n/2}.
	\]

The problems we study originate in the spectral theory of symmetric stable processes. If $X$ is the rotationally symmetric $2s$-stable process killed upon leaving $\Omega$, then the negative of the generator of its killed semigroup is the exterior Dirichlet realization in \eqref{eq:eigenproblem}. Its ground state governs, among other things, the long-time behaviour of the killed process. While many properties of the eigenvalues and eigenfunctions have been investigated, basic questions about the shape of the ground state have remained open; see \cite{BBKRSV80,ChenSong2005,BaKu06,BaKu08,Kwa12,FrGe16,DKK17} and the references therein.

One such question was raised by Ba\~nuelos, Kulczycki, and M\'endez-Hern\'andez in \cite[Conjecture~1.1]{BanuelosKulczyckiMendez2006}:
\begin{equation}\label{eq:BKMH-conjecture}
    \text{the ground state in $(-1,1)$ is concave for every $s\in(0,1)$.}
\end{equation}
In higher dimensions, the corresponding question asks whether the ground state is superharmonic, a property which is equivalent to concavity in one dimension.

The first result in this direction is due to Ba\~nuelos and Kulczycki, who proved \eqref{eq:BKMH-conjecture} in the Cauchy case $s=1/2$ by relating the problem to a mixed Steklov problem \cite[Theorem~4.7]{BanuelosKulczycki2004}, cf.\ Caffarelli--Silvestre extension, \cite{caffarelli2007extension}. Ba\~nuelos, Kulczycki, and M\'endez-Hern\'andez introduced a weaker notion, called mid-concavity, and established it for every symmetric stable process in intervals and, more generally, in rectangles of all dimensions \cite{BanuelosKulczyckiMendez2006}. In their 2014 note \cite{KassmannSilvestre2014}, Kassmann and Silvestre proved superharmonicity with analytic tools in bounded $C^{1,1}$-domains when $1/s\in\N$. The main idea is a repeated application of the fractional Laplacian until the order reaches the ordinary Laplacian. Ba\~nuelos and DeBlassie gave a probabilistic proof on bounded Lipschitz domains for $s=1/m$ with integers $m>2$ \cite[Theorem~1.1]{BanuelosDeBlassie2015}. These arguments therefore apply to a discrete family of exponents. Abatangelo and Jarohs proved \eqref{eq:BKMH-conjecture} for $s\in(1/2,1)$ by a computer-assisted analysis of explicit potentials \cite{AbatangeloJarohs2024}. Their method also gives superharmonicity in balls for $s\in(1/2,1)$ in dimensions $2\le n\le11$. It relies on delicate estimates of the available potentials for the fractional Laplacian in balls; see also \cite{Buc16}. Both the restriction on $s$ and the dimensional threshold enter through the estimates of the corresponding kernels.

Our main result, \cref{thm:superharmonicity}, settles \eqref{eq:BKMH-conjecture}. In fact, whenever a non-negative eigenfunction exists, it is strictly superharmonic, with a quantitative pointwise lower bound, on an arbitrary open set throughout the full range $s\in(0,1)$.

	\begin{theorem}\label{thm:superharmonicity}

Let $\Omega\subset\R^n$ be open, $s\in(0,1)$, and let us assume that a weak eigenpair $(\lambda,u)\in \R\times \left(H_0^s(\Omega)\setminus\{0\}\right)$ exists with $u \ge 0$. Then $\lambda>0$, $u\in C^\infty_{\mathrm{loc}}(\Omega)$, $u > 0$ in $\Omega$, and 
\begin{equation}\label{eq:full-range-superharmonicity}
    -\Delta u>\lambda^{1/s}u>0,\qquad
    -\Delta\log u>\lambda^{1/s}+|\nabla\log u|^2
    \quad\text{in }\Omega.
\end{equation}

	\end{theorem}

Let us comment on the assumptions in \cref{thm:superharmonicity}. Weak eigenpairs are understood as in \cref{def:weak-eigenpair}. By the Picone inequality \cite[Proposition~4.2]{BrascoFranzina2014}, a non-negative eigenfunction as in \cref{thm:superharmonicity} attains
\begin{equation}\label{eq:first-eigenvalue-variational}
    \lambda_s(\Omega)\coloneqq
    \inf_{v\in H_0^s(\Omega)\setminus\{0\}}
    \frac{\langle v,v\rangle_s}{\|v\|_{L^2(\Omega)}^2}.
\end{equation}
Thus $\lambda=\lambda_s(\Omega)$, and $u$ is a multiple of the non-negative, $L^2$-normalized minimizer, called the \emph{ground state} \cite[Proposition~3.4]{FranzinaLicheri2022}. Existence is not automatic on arbitrary open sets (see \cite[Lemma~1]{Kwa09} for a compactness criterion); for instance, $\R^n$ has no nonzero $L^2$ eigenfunction. If $0<|\Omega|<\infty$, compactness of $H_0^s(\Omega)\hookrightarrow L^2(\Omega)$ gives a ground state and $\lambda_s(\Omega)>0$ \cite[(2.2) and Theorem~2.1]{FernandezBonderRitortoSalort2017} and \cite[Proposition~3.2]{FranzinaLicheri2022}.

	In one dimension, \cref{thm:superharmonicity} immediately yields the
	following consequence.

	\begin{corollary}\label{cor:one-dimensional}
		Under the hypotheses of \cref{thm:superharmonicity}, let $n=1$. Then,
		on every connected component of $\Omega$,
		\[
		u''<-\lambda^{1/s}u<0,
		\qquad
		(\log u)''<-\lambda^{1/s}.
		\]
		Thus $u$ is strictly concave and $\log u$ is
		$\lambda^{1/s}$-strongly concave.
	\end{corollary}

	In particular, \cref{cor:one-dimensional} settles the conjecture raised in \cite[Conjecture~1.1]{BanuelosKulczyckiMendez2006},  \eqref{eq:BKMH-conjecture}. The estimate \eqref{eq:full-range-superharmonicity} belongs to a hierarchy comparing different powers of the Laplacian.

	\begin{theorem}\label{thm:power}
Under the hypotheses of \cref{thm:superharmonicity}, for every $k\in\mathbb N_0$ we have pointwise in $\Omega$
		\begin{align}
			\Delta^k\bigl[(\lambda^\theta-(-\Delta)^{s\theta})u\bigr]&>0
			&&\text{if }0<\theta<1,
			\label{eq:iterated-low-comparison}\\
			\Delta^k\bigl[((-\Delta)^{s\theta}-\lambda^\theta)u\bigr]&>0
			&&\text{if }1<\theta\le1+\frac1s.
			\label{eq:iterated-subharmonicity}
		\end{align}
	\end{theorem}

Taking $\theta=1/s$ and $k=0$ in \eqref{eq:iterated-subharmonicity} gives the first inequality in \eqref{eq:full-range-superharmonicity}.

The pointwise estimate \eqref{eq:full-range-superharmonicity} also gives a strict comparison with the first classical Dirichlet eigenvalue.

	\begin{corollary}\label{cor:eigenvalue-comparison}
		Let $\Omega\subset\R^n$ be an open set with finite measure, $s\in(0,1)$, and $u$ be the fractional ground state. If $\mu_1(\Omega)$ is the first	Dirichlet eigenvalue of $-\Delta$ on $\Omega$, then
		\[
		      \inf_{x\in\Omega}\frac{-\Delta u(x)}{u(x)}\le\mu_1(\Omega),
		      \qquad
		      \lambda_s(\Omega)<\mu_1(\Omega)^s.
		\]
	\end{corollary}

Chen and Song proved the non-strict eigenvalue comparison for every eigenvalue on bounded open sets \cite[Theorem~3.4 and Remark~3.5(i)]{ChenSong2005}. Strictness was proved for the first eigenvalue on smooth bounded domains by Servadei and Valdinoci \cite[Proposition~5.1]{ServadeiValdinoci2014}, and for the entire spectrum on bounded Lipschitz domains by Musina and Nazarov \cite[Corollary~4]{MusinaNazarov2014}.  We nevertheless include	\cref{cor:eigenvalue-comparison} because here it is an immediate consequence of the pointwise estimate in \cref{thm:superharmonicity}.

	\subsection*{Complete Bernstein functions}

The method also applies to more general operators than the fractional Laplacian. Instead of taking the $s$-th power of the Laplacian, we may consider a complete Bernstein function $\psi$ and study $\psi(-\Delta)$ (see \Cref{sec:function-spaces}). Probabilistically, this corresponds to replacing the stable subordinator by a general complete subordinator; see \cite{SchillingSongVondracek2012}. This leads us to the exterior Dirichlet problem
    \begin{equation}
        \label{eq:eigenproblem-bernstein}
        \begin{cases}
            \psi(-\Delta)u = \lambda u, & \text{ in } \Omega,\\
            \phantom{\psi(-\Delta)}u =0, & \text{ on } \Omega^c.
        \end{cases}
    \end{equation}
	
	\begin{theorem}\label{thm:general-Bernstein}

Let $\Omega\subset\R^n$ be open, and let $\psi$ be a non-affine complete Bernstein function. If $(\lambda, u)\in (0,\infty)\times H_0^\psi(\Omega)\setminus\{0\}$ is a weak eigenpair of \eqref{eq:eigenproblem-bernstein} with $u\ge 0$, then $u\in C_{\text{loc}}^\infty(\Omega)$, $u>0$ in $\Omega$, and
		\[
		      -\Delta u\ge\psi^{-1}(\lambda)u>0\qquad\text{in }\Omega.
		\]
		More generally, for every $k \in \mathbb N_0$, 
		\begin{equation}\label{eq:general-hierarchy}
			\Delta^k\bigl[(-\Delta-\psi^{-1}(\lambda))u\bigr]\ge 0
			\qquad\text{in }\Omega,
		\end{equation}
		with strict inequality if $\psi$ is unbounded.
	\end{theorem}

	Thus the superharmonicity principle is a consequence of the Bernstein
	structure rather than a peculiarity of stable processes. The assumption
	that $\psi$ is non-affine ensures positivity on an arbitrary, possibly
	disconnected, open set through the strong maximum principle. If $\Omega$ is
	connected, it is enough to assume that $\psi$ is non-constant.

    The existence of the ground state $u$ in \cref{thm:general-Bernstein} can be ensured if $\Omega$ has finite measure and if $\psi$ is unbounded due to the compact embedding $H_0^\psi(\Omega)\hookrightarrow L^2(\Omega)$, cf.\ \cite{JaWe20}.

    \subsection*{The fractional Brascamp--Lieb problem}

The conjecture \eqref{eq:BKMH-conjecture} belongs to a broader question raised in \cite{BanuelosKulczyckiMendez2006}. The classical Brascamp--Lieb theorem states that the ground state of $-\Delta$ on every bounded convex domain is log-concave \cite[Theorem~6.1]{BrascampLieb1976}. Besides its intrinsic geometric content, this property has applications in analysis, geometry, mathematical physics, and probability. In particular, log-concavity leads to estimates of the spectral gap $\mu_2-\mu_1$ between the first and second Dirichlet eigenvalues, see, e.g., \cite{Kwa08, Kal12}. This gap describes the rate to equilibrium of Brownian motion conditioned to remain in the domain; see \cite{SWYY85} and \cite[p.~205]{BanuelosKulczyckiMendez2006}.

Motivated by the classical Brascamp--Lieb theorem, Ba\~nuelos, Kulczycki, and M\'endez-Hern\'andez raised the following question in \cite{BanuelosKulczyckiMendez2006}:

\medskip
\emph{Given a bounded convex domain $\Omega\subset\R^n$ and $s\in(0,1)$, is the ground state of \eqref{eq:eigenproblem} log-concave?}

\medskip 

One probabilistic proof of the classical Brascamp--Lieb theorem proceeds through the log-concavity of the finite-dimensional distributions of Brownian motion. For stable processes this argument breaks down: their transition densities are not log-concave, and subordination does not preserve log-concavity in the required form \cite{BanuelosKulczyckiMendez2006}.

\Cref{cor:one-dimensional} gives a positive answer in intervals. For balls in every dimension, the higher-order inequalities in \cref{thm:power} lead to a quantitative Hessian estimate. Radial symmetry reduces the problem to two one-dimensional inequalities, and \eqref{eq:general-hierarchy} gives the same conclusion for ground states of $\psi(-\Delta)$.

\begin{corollary}\label{cor:ball-log-concavity}
Let $n\in\N$, $R>0$, and let $\psi$ be a non-affine complete Bernstein function. If $(\lambda, u)\in [0,\infty)\times H_0^\psi(B_R)\setminus\{0\}$ is a weak eigenpair in $B_R$ with $u\ge 0$, then
\begin{equation}\label{eq:ball-log-concavity}
    D^2\log u(x)\le D^2\log u(0)
    \le-\frac{\psi^{-1}(\lambda)}{n}\,\Id
    \quad\text{for }x\in B_R.
\end{equation}
The first inequality\footnote{Here and below,  matrix inequalities are understood in the sense of quadratic forms: $A\le B$ means that $B-A$ is positive semidefinite, and $A<B$ means that it is positive definite.} is strict for $x\ne0$, and the second if $\psi$ is unbounded.
\end{corollary}

With $\psi(t)=t^s$, \cref{cor:ball-log-concavity} gives strict log-concavity of the fractional ground state in every ball, for every $s\in(0,1)$ and in every dimension.

The strict Hessian estimate \eqref{eq:ball-log-concavity} raises a further question: \textit{does log-concavity persist under sufficiently small convex $C^{1,\alpha}$-perturbations of the ball?} For general bounded convex domains, however, the fractional Brascamp--Lieb property is false. \Cref{thm:counter-example} gives counterexamples in sufficiently thin ellipsoids.

    \begin{proposition}
        \label{thm:counter-example}
        Let $s\in(0,1)$ and $n\ge2$. For all sufficiently small
        $\varepsilon>0$, the ground state of $(-\Delta)^s$ in the ellipsoid
        \[
            \Omega_\varepsilon
            \coloneqq
            \left\{(t,z)\in\R\times\R^{n-1}:
            t^2+\frac{|z|^2}{\varepsilon^2}<1\right\}
        \]
        is not log-concave.
    \end{proposition}

\Cref{thm:counter-example} gives smooth bounded strictly convex counterexamples. The failure occurs in the interior: since the ground state is smooth and positive inside $\Omega_\varepsilon$, its logarithm has a positive second directional derivative at some interior point. In the proof of \cref{thm:counter-example}, concentration near the middle of the ellipsoid would force exponentially small tails under log-concavity, whereas the nonlocal equation gives a polynomial lower bound.

\subsection*{Idea of the proof}
Let us sketch the proof of the superharmonicity of the ground state in the fractional case, \cref{thm:superharmonicity}. The initial observation is that
\[
    \mu\coloneqq\lambda u-(-\Delta)^s u
\]
is a nonzero non-negative Radon measure supported in $\Omega^c$. The eigenvalue equation \eqref{eq:eigenproblem} gives $\mu=0$ in $\Omega$. On the other hand,
\[
    (-\Delta)^s u(x)=-c_{n,s}\int_\Omega
    \frac{u(y)}{|x-y|^{n+2s}}\,\mathrm dy<0
    \quad\text{for }x\notin\overline\Omega.
\]
Thus, the eigenfunction $u$ solves a whole-space equation with a positive defect concentrated outside the domain. The rigorous statement, including any contribution on $\partial\Omega$, is established in \cref{prop:general-defect}.

To see how this exterior defect produces an inequality inside $\Omega$, consider the factorization
\begin{equation}\label{eq:factorization-formal}
    (\lambda^{1/s}-(-\Delta))u=G(-\Delta)\mu,
    \qquad G(r)=\frac{\lambda^{1/s}-r}{\lambda-r^s}.
\end{equation}
The essential observation, proved in \cref{prop:G-is-Bernstein}, is that $G$ is a complete Bernstein function. Hence $-G(-\Delta)$, like $-(-\Delta)^{1/2}$, satisfies the positive maximum principle; see \cite{Courrege1965,Sch01}. The sign needed here is made explicit by the representation
\[
    G(r)=\lambda^{(1-s)/s}+\int_0^\infty\frac{r}{r+t}\,\sigma(\mathrm dt)
\]
for a nonzero positive measure $\sigma$. For $t>0$, let $v_t\coloneqq(-\Delta+t)^{-1}\mu$. The Bessel kernel is strictly positive, and therefore $v_t>0$ in $\Omega$. Since $\mu=0$ there,
\[
    (-\Delta+t)v_t=0\quad\text{in }\Omega,\qquad
    \Delta^k v_t=t^k v_t>0\quad\text{for }k\in\mathbb N_0.
\]
Moreover, $(-\Delta)(-\Delta+t)^{-1}\mu=\mu-tv_t=-tv_t$ in $\Omega$. The local terms in the Bernstein representation vanish there, so \eqref{eq:factorization-formal} gives
\[
    -\Delta u-\lambda^{1/s}u=\int_0^\infty t v_t\,\sigma(\mathrm dt)>0.
\]
Applying $\Delta^k$ replaces $v_t$ by $t^kv_t$ and preserves positivity.

The same mechanism yields the full range in \cref{thm:power}. By \cref{prop:power-quotient}, the quotient
\[
    \Psi(r)=\frac{r^{s\theta}-\lambda^\theta}{r^s-\lambda}
\]
is Stieltjes when $0<\theta<1$ and complete Bernstein when $1<\theta\le1+1/s$. These two representations give the opposite signs on the two sides of $\theta=1$.

For \cref{cor:ball-log-concavity}, radial symmetry reduces the Hessian estimate to two one-dimensional inequalities for the radial and tangential eigenvalues of $D^2\log u$. In the fractional case, \cref{thm:power} with $\theta=1/s$ and $k=0,1$ gives $g,\Delta g\ge0$ for $g=(-\Delta-\lambda^{1/s})u$. For general $\psi$, \cref{thm:general-Bernstein} gives the same inequalities with $g=(-\Delta-\psi^{-1}(\lambda))u$. They control both Hessian eigenvalues through the radial equation.

	\subsection*{Why direct iteration is insufficient}

It is useful to compare the factorization \eqref{eq:factorization-formal} with the finite-iteration method of \cite{KassmannSilvestre2014}. In terms of the exterior defect from \cref{prop:general-defect}, the whole-space equation is
\begin{equation}\label{eq:iteration-whole-space}
    ((-\Delta)^s-\lambda)u=-\mu.
\end{equation}
Suppose first that $s=1/2$. After regularizing \eqref{eq:iteration-whole-space}, one may apply the half Laplacian once more. Since $\mu$ is non-negative and supported in $\Omega^c$, its half Laplacian has the opposite sign at interior points in $\Omega$. Thus, formally,
\[
    -\Delta u-\lambda^2u=-(-\Delta)^{1/2}\mu>0
    \qquad\text{in }\Omega.
\]
This is the basic positivity mechanism in the half-Laplacian argument.

For a general $s$, however, no finite number of applications of $(-\Delta)^s$ reaches the ordinary Laplacian unless $1/s$ is an integer; see \cref{rem:integer-iteration}. For example, when $s>1/2$, applying the operator once more still gives the formally useful inequality
\[
    (-\Delta)^{2s}u-\lambda^2u=-(-\Delta)^s\mu>0
    \qquad\text{in }\Omega.
\]
One might then try to use
\[
    (-\Delta)^{2s}u=(-\Delta)^{2s-1}(-\Delta u)
\]
and apply a weak maximum principle for $(-\Delta)^{2s-1}$ to $-\Delta u$.

The issue lies at the boundary $\partial\Omega$, even though $u$ is smooth in $\Omega$. In a smooth domain, the natural boundary profile $\operatorname{dist}(x,\partial\Omega)^s$, has a second normal derivative of size $\operatorname{dist}(x,\partial\Omega)^{s-2}$, which is not locally integrable in the normal direction. The distributional second normal derivative of its zero extension is therefore not a Radon measure. Thus interior smoothness does not justify regarding $-\Delta u$ as an admissible function with zero exterior values, as required by this weak maximum-principle argument.

The factorization \eqref{eq:factorization-formal} avoids this difficulty by
\[
    -\Delta u-\lambda^{1/s}u=-G(-\Delta)\mu.
\]
The key difference is that $G(-\Delta)$ acts directly on the Radon measure $\mu$, rather than on the singular global distribution $-\Delta u$. Its complete Bernstein representation, given by \cref{prop:G-is-Bernstein}, expresses the right-hand side as a positive integral of Bessel potentials in $\Omega$.

\section{The exterior defect}\label{sec:defect}

    \subsection{Function spaces \& weak solutions}\label{sec:function-spaces}

We introduce the Sobolev-like function spaces associated to Bernstein functions $\psi$ and define the weak solution concept.

  A function $\psi:(0,\infty)\to[0,\infty)$ is a \emph{Bernstein function} if $\psi\in C^\infty(0,\infty)$ and
	\[
	   (-1)^{k-1}\psi^{(k)}(t)\ge0 \qquad\text{for every }t>0\text{ and }k\in\mathbb N.
	\]
	Equivalently, there are unique constants $c,b\ge0$ and a unique positive measure $\nu$ on $(0,\infty)$ such that
	\begin{equation}\label{eq:Bernstein-Levy-Khintchine}
		\psi(\tau)=c+b\tau+
		\int_{0}^{\infty}(1-e^{-t\tau})\,\nu(\mathrm dt),
		\qquad
		\int_{0}^{\infty}\min\{1, t\}\,\nu(\mathrm dt)<\infty.
	\end{equation}
	A Bernstein function is	\emph{complete} if $\nu$ is absolutely continuous with a completely
	monotone density; equivalently, it admits the representation \cref{eq:CBF-general} below. Every non-affine Bernstein function is strictly increasing and, thus, is invertible on its image. We refer to \cite[Theorem~3.2, Definition~6.1, and Theorem~6.2(vi)]{SchillingSongVondracek2012} for more details.

    We use the Fourier convention
	\[
	   \mathcal F(v)(\xi)\coloneqq \int_{\R^n}v(x)e^{-i\xi\cdot x}\,\mathrm dx.
	\]

For a non-affine complete Bernstein function $\psi$, set

	\[
        H^\psi(\R^n)\coloneqq \left\{v\in L^2(\R^n): \int_{\R^n}\psi(|\xi|^2)|\mathcal F(v)(\xi)|^2\,\mathrm d\xi<\infty	\right\}
	\]
    and define the symmetric form
	\begin{equation}\label{eq:psi-form}
		\langle u,v\rangle_\psi	\coloneqq\frac1{(2\pi)^n}\int_{\R^n}\psi(|\xi|^2)\mathcal F(u)(\xi)\overline{\mathcal F(v)(\xi)}\,\mathrm d\xi.
	\end{equation}

With the norm $\|v\|_{H^\psi}^2=\|v\|_2^2+\langle v,v\rangle_\psi$, this is a Hilbert space and $C_c^\infty(\R^n)$ is dense by \cite[Proposition~2.2.4]{FJS01} and cutoff approximation.  In the case $\psi(t)= t^s$, the space becomes the classical fractional Sobolev space $H^s(\R^n)$ and we simply write $\langle u,v\rangle_s$ for the associated form. Note that this definition coincides with 
\[
    \langle u,v\rangle_s=\frac{c_{n,s}}2
    \iint_{\R^n\times\R^n}
    \frac{(u(x)-u(y))(v(x)-v(y))}{|x-y|^{n+2s}}\,\mathrm dx\,\mathrm dy,
\]
see \cite{FernandezRealRosOton2024}. For an open set $\Omega$, let

	\[
	   H_0^\psi(\Omega)\coloneqq \overline{C_c^\infty(\Omega)}^{\,H^\psi(\R^n)}.
	\]

With $p_t(z)=(4\pi t)^{-n/2}e^{-|z|^2/(4t)}$, \eqref{eq:Bernstein-Levy-Khintchine} and Plancherel's identity give
\begin{equation}\label{eq:psi-Beurling-Deny}
\begin{aligned}
    \langle v,v\rangle_\psi
    &=\psi(0)\|v\|_2^2+b\|\nabla v\|_2^2 +\frac12\int_0^\infty\iint_{\R^n\times\R^n}
    p_t(x-y)|v(x)-v(y)|^2\,\mathrm dx\,\mathrm dy\,\nu(\mathrm dt).
\end{aligned}
\end{equation}

Thus $\langle\cdot,\cdot\rangle_\psi$ is a symmetric regular Dirichlet form. Moreover, $H_0^\psi(\Omega)$  is closed under taking smaller nonnegative functions, \cite[Section~4.4]{FukushimaOshimaTakeda2011}:
\begin{equation}\label{eq:psi-order-ideal}
    0\le v\le w,\quad v\in H^\psi(\R^n),\quad
    w\in H_0^\psi(\Omega)\quad\Longrightarrow\quad
    v\in H_0^\psi(\Omega).
\end{equation}

The multiplier $\psi(-\Delta)$ is defined by
$\mathcal F(\psi(-\Delta)v)(\xi)=\psi(|\xi|^2)\mathcal F(v)(\xi)$.
Writing $H^{-\psi}$ for the dual of $H^\psi$ with respect to $L^2$, we have
$\langle\psi(-\Delta)w,\varphi\rangle=\langle w,\varphi\rangle_\psi$;
in particular, $\psi(-\Delta):H^\psi\to H^{-\psi}$ is continuous
\cite[Theorem~2.2.10]{FJS01}.

    These function spaces allow us to setup the weak solution concept for solutions to \eqref{eq:eigenproblem-bernstein}.

    \begin{definition}\label{def:weak-eigenpair}
        We say that	$(\lambda,u)\in\R\times(H_0^\psi(\Omega)\setminus\{0\})$ is a \emph{weak eigenpair} if
	\begin{equation}\label{eq:weak-psi-eigen}
		\langle u,\varphi\rangle_\psi=\lambda\int_\Omega u\varphi\,\mathrm dx \qquad\text{for every }\varphi\in H_0^\psi(\Omega).
	\end{equation}
    \end{definition}

Set
\begin{equation}\label{eq:first-psi-eigenvalue}
    \lambda_\psi(\Omega)\coloneqq
    \inf_{v\in H_0^\psi(\Omega)\setminus\{0\}}
    \frac{\langle v,v\rangle_\psi}{\|v\|_{L^2(\Omega)}^2}.
\end{equation}
A non-negative, $L^2$-normalized minimizer is called a \emph{ground state}.
If $0<|\Omega|<\infty$ and $\psi$ is unbounded, compactness of
$H_0^\psi(\Omega)\hookrightarrow L^2(\Omega)$ gives a minimizer \cite{JaWe20} (with classical compactness when $b>0$),
which may be chosen non-negative by the Markov property (replacing a minimizer $u$ by $|u|$ does not increase its energy, while preserving its $L^2$-norm). Since $\psi$ is non-affine, $\nu\ne0$ in \eqref{eq:Bernstein-Levy-Khintchine}, and the jump kernel
\[
    J(z)\coloneqq\int_0^\infty p_t(z)\,\nu(\mathrm dt)>0
    \quad\text{for}\quad z\ne0.
\]
couples every two subsets of positive measure. The usual variational argument gives simplicity of the first eigenvalue; positivity follows in the proof of \cref{thm:general-Bernstein}.

    \subsection{An auxiliary Radon measure}

    As seen in the idea of the proof, the key auxiliary quantity is the tempered distribution $\lambda u-\psi(-\Delta)u$. One observes that eigenvalue equation enforces this to be zero in $\Omega$. The next proposition ensures that the tempered distribution has a sign and, thus, becomes a Radon-measure.

    \begin{proposition}\label{prop:general-defect}

Let $\Omega\subset\R^n$ be open, $\psi$ be a non-affine complete Bernstein function, and $(\lambda,u)$ be a weak eigenpair with $u\ge0$. Then, 
\[
\mu\coloneqq\lambda u-\psi(-\Delta)u
\]
is a nonzero non-negative tempered Radon measure in $H^{-\psi}(\R^n)$, supported in $\Omega^c$, and
\begin{equation}\label{eq:general-defect-mass}
    \psi(0)<\lambda<\psi(\infty),\qquad
    \mu(\R^n)=(\lambda-\psi(0))\int_{\R^n}u.
\end{equation}

	\end{proposition}

We recall that a non-negative measure $\mu$ is tempered if $\int(1+|x|)^{-N}\,\mu(\mathrm dx)<\infty$ for some $N\in\N$.

	\begin{proof}

Testing \eqref{eq:weak-psi-eigen} with $u$ gives
$\lambda\|u\|_2^2=\langle u,u\rangle_\psi$, hence
$\psi(0)<\lambda<\psi(\infty)$ by \eqref{eq:psi-form}.
Indeed, $\psi(0)<\psi(t)<\psi(\infty)$ for $t>0$, and $\mathcal F(u)$ is non-zero on a set of positive measure.
Continuity of $\psi(-\Delta):H^\psi\to H^{-\psi}$ gives $\mu\in H^{-\psi}$,
and \eqref{eq:weak-psi-eigen} gives $\supp\mu\subset\Omega^c$.

To prove $\mu\ge0$, we use a truncation to obtain an admissible test function in \eqref{eq:weak-psi-eigen}.
Let $0\le\eta\in C_c^\infty(\R^n)$, $M>0$, and $q=(\eta-Mu)^+$.
By \eqref{eq:psi-order-ideal}, $\eta-q=\min\{\eta,Mu\}\in H_0^\psi(\Omega)$.
Testing \eqref{eq:weak-psi-eigen} with $\eta-q$ yields
\begin{equation}\label{eq:general-defect-test}
    \langle\mu,\eta\rangle=\lambda\int_{\R^n}uq-\langle u,q\rangle_\psi.
\end{equation}
Given that $\langle f^+,f^-\rangle_\psi\le0$ for \(f\in H^\psi(\mathbb R^n)\) by the Markov property, taking \(f=\eta-Mu\) gives
$\langle \eta-Mu,q\rangle_\psi\ge \langle q,q\rangle_\psi$. Since $u,q\ge0$ and $\lambda>0$, completing the square in \eqref{eq:general-defect-test} gives
\[
    \langle\mu,\eta\rangle\ge
    \frac{\langle q-\eta/2,q-\eta/2\rangle_\psi-\frac14\langle\eta,\eta\rangle_\psi}{M}
    \ge-\frac{\langle\eta,\eta\rangle_\psi}{4M}.
\]
Letting $M\to\infty$ proves that $\mu$ is a non-negative Radon measure
\cite[Theorem~2.1.7]{HormanderI}.

It remains to prove the mass identity in \eqref{eq:general-defect-mass}.
If $u\in L^1(\R^n)$, choose a radial non-increasing cutoff
$0\le\chi\in C_c^\infty(\R^n)$ with $\chi=1$ in $B_1$ and set $\chi_R(x)=\chi(x/R)$. Then $\chi_R\uparrow1$.
For $P_t=e^{t\Delta}$, the Phillips formula \cite[Theorem~13.6]{SchillingSongVondracek2012} gives
\[
    \psi(-\Delta)\chi_R=\psi(0)\chi_R-b\Delta\chi_R
    +\int_0^\infty(\chi_R-P_t\chi_R)\,\nu(\mathrm dt).
\]
The heat semigroup is an $L^\infty$ contraction, and $P_t\chi_R-\chi_R=\int_0^tP_\tau\Delta\chi_R\,\mathrm d\tau$. For fixed $t$, both $\chi_R$ and $P_t\chi_R$ tend pointwise to one.
Since $\|\Delta\chi_R\|_\infty\le CR^{-2}$ and
$\|P_t\chi_R-\chi_R\|_\infty\le C\min\{1,t\}$ for $R\ge1$,
we have $\psi(-\Delta)\chi_R\to\psi(0)$ pointwise with a uniform bound.
Self-adjointness, dominated convergence against $u\,\mathrm dx$, and monotone convergence against $\mu$ give
\[
    \mu(\R^n)=\lim_{R\to\infty}\langle\mu,\chi_R\rangle
    =(\lambda-\psi(0))\int_{\R^n}u.
\]
Conversely, if $\mu$ is finite, its Fourier transform is continuous. Since $\lambda>\psi(0)$, the denominator $\lambda-\psi(|\xi|^2)$ stays away from zero near the origin. Thus
$\widehat u=\widehat\mu/(\lambda-\psi(|\xi|^2))$ has a continuous representative near zero.
Gaussian regularization and monotone convergence give
\[
    \int_{\R^n}u=\lim_{\varepsilon\downarrow0}
    \int_{\R^n}u(x)e^{-\varepsilon|x|^2}\,\mathrm dx
    =\frac{\mu(\R^n)}{\lambda-\psi(0)}<\infty.
\]
Indeed, by Plancherel's theorem the regularized integrals are Gaussian averages of $\widehat u$, with kernels $(4\pi\varepsilon)^{-n/2}e^{-|\xi|^2/(4\varepsilon)}$ of integral one. Continuity controls a neighborhood of zero, and Cauchy--Schwarz controls the complement using $\widehat u\in L^2$: the kernels' $L^2$ norm outside any fixed neighborhood tends to zero.
This proves \eqref{eq:general-defect-mass}, also when both sides are infinite,
and implies $\mu\ne0$, since $\lambda-\psi(0)>0$ and $u\ge0$ is non-zero.

	\end{proof}

\section{Bernstein functions and Bessel potentials}
\label{sec:Bernstein-resolvents}

A complete Bernstein function has the representation
\begin{equation}\label{eq:CBF-general}
    F(r)=a+br+\int_0^\infty\frac{r}{r+t}\,\sigma(\mathrm dt),
    \qquad a,b\ge0,\qquad
    \int_0^\infty\frac{\sigma(\mathrm dt)}{1+t}<\infty.
\end{equation}
A function $f:(0,\infty)\to(0,\infty)$ is Stieltjes if $rf(r)$ is complete Bernstein. These representations are unique; see \cite[Chapters~6--7]{SchillingSongVondracek2012}.

The following  is a standard consequence of the integral representation of complete Bernstein functions, cf. \cite[Theorem~6.2(vi)]{SchillingSongVondracek2012}; and the conjugation property for complete Bernstein functions \cite[Proposition~7.1 and Remark~7.2]{SchillingSongVondracek2012}.

\begin{proposition}\label{prop:G-is-Bernstein}
    Let $\psi$ be a nonconstant complete Bernstein function and $a>0$. Then
    \begin{equation}\label{eq:general-G-definition}
        G(t)\coloneqq\frac{t-a}{\psi(t)-\psi(a)},
        \qquad G(a)=\frac1{\psi'(a)},
    \end{equation}
    is complete Bernstein. If $\psi$ is unbounded and non-affine, the representing measure of $G$ in \eqref{eq:CBF-general} is nonzero.
\end{proposition}

\begin{proof}
    Let $h(t)=\alpha+\beta t+\int_0^\infty\frac{t}{t+r}\,\rho(\mathrm dr)$ be complete Bernstein. Since $\frac{t}{t+r}-\frac{a}{a+r}=\frac{r(t-a)}{(t+r)(a+r)}$, the divided difference
    \begin{equation}\label{eq:CBF-divided-difference}
        t\frac{h(t)-h(a)}{t-a}
        =\beta t+\int_0^\infty\frac{t}{t+r}\frac{r}{a+r}\,\rho(\mathrm dr)
    \end{equation}
    is complete Bernstein, with value $ah'(a)$ at $t=a$, since $0\le r/(a+r)\le1$. By \cite[Proposition~7.1]{SchillingSongVondracek2012}, $t/F(t)$ is complete Bernstein whenever $F\not\equiv0$ is complete Bernstein. Applying this to \eqref{eq:CBF-divided-difference} with $h=\psi$ shows that $G$ is complete Bernstein. Moreover, $G(0)=a/(\psi(a)-\psi(0))$.

    If the representing measure of $G$ in \eqref{eq:CBF-general} vanishes, then $G(t)=A+Bt$ and $\psi(t)=\psi(a)+(t-a)/(A+Bt)$. For $B>0$, the expression for $\psi(t)$ tends to the finite limit $\psi(a)+1/B$ as $t\to\infty$. Unboundedness therefore forces $B=0$, so $\psi$ is affine.
\end{proof}

The following quotient gives the comparison in \cref{thm:power}.

\begin{proposition}\label{prop:power-quotient}
    Let $s\in(0,1)$, $\lambda>0$, and $0<\theta\le1+1/s$. For $r\ge0$, define
    \begin{equation}\label{eq:Psi-definition}
        \Psi(r)\coloneqq\frac{r^{s\theta}-\lambda^\theta}{r^s-\lambda},
        \qquad \Psi(\lambda^{1/s})=\theta\lambda^{\theta-1}.
    \end{equation}
    If $0<\theta<1$, there is a unique nonzero positive measure $\sigma$ on $(0,\infty)$ such that
    \begin{equation}\label{eq:Psi-low-Stieltjes}
        \Psi(r)=\int_0^\infty\frac{\sigma(\mathrm dt)}{r+t},
        \qquad \int_0^\infty\frac{\sigma(\mathrm dt)}t=\lambda^{\theta-1}.
    \end{equation}
    If $1\le\theta\le1+1/s$, then $\Psi$ is complete Bernstein and its measure in \eqref{eq:CBF-general} is zero if and only if $\theta=1$.
\end{proposition}

\begin{proof}
    For $0<\theta<1$, \cref{prop:G-is-Bernstein} applied to $t^\theta$ shows that $(t-\lambda)/(t^\theta-\lambda^\theta)$ is complete Bernstein. The reciprocal of a non-zero complete Bernstein function is Stieltjes, and the composition of a Stieltjes function with a non-zero complete Bernstein function is Stieltjes; see \cite[Chapters~6--7]{SchillingSongVondracek2012}. Since $r\mapsto r^s$ is complete Bernstein, $\Psi$ is Stieltjes. The coefficients of the $1/r$ and constant terms are $\lim_{r\downarrow0}r\Psi(r)$ and $\lim_{r\uparrow\infty}\Psi(r)$, respectively. Since $\Psi(0)=\lambda^{\theta-1}$ and $\Psi(\infty)=0$, its representation has no $1/r$ or constant term, and monotone convergence gives \eqref{eq:Psi-low-Stieltjes}. In particular, $\sigma\ne0$.

    The case $\theta=1$ gives $\Psi=1$. For $1<\theta\le2$, write
    \[
        \Psi(r)=\lambda^{\theta-1}+
        \left.t\frac{t^{\theta-1}-\lambda^{\theta-1}}{t-\lambda}\right|_{t=r^s}.
    \]
    Here $t^{\theta-1}$ is complete Bernstein because $0<\theta-1\le1$. By \eqref{eq:CBF-divided-difference} and closure under composition, this is complete Bernstein. If $2<\theta\le1+1/s$, set $\alpha=s(\theta-1)\in(s,1]$ and write
    \[
        \Psi(r)=r^\alpha+\lambda\frac{r^\alpha-\lambda^{\theta-1}}{r^s-\lambda}.
    \]
    Both $r^\alpha$ and $t^{s/\alpha}$ are complete Bernstein, since $0<s<\alpha\le1$.
    The quotient is complete Bernstein by \cref{prop:G-is-Bernstein} applied to $t^{s/\alpha}$, followed by composition with $r^\alpha$; see \cite[Corollary~7.6]{SchillingSongVondracek2012}.

    For $\theta\ge1$, the coefficients in \eqref{eq:CBF-general} satisfy $a=\Psi(0)=\lambda^{\theta-1}$ and $b=\lim_{r\to\infty}\Psi(r)/r$: thus $b=0$ for $\theta<1+1/s$ and $b=1$ at $\theta=1+1/s$.
    If $\sigma=0$, then $\Psi(r)=a+br$; we exclude this for $\theta>1$ by examining the remaining growth.
    If $1<\theta<1+1/s$, then $b=0$ and $\Psi(r)\sim r^{s(\theta-1)}\to\infty$, so $\sigma\ne0$. At $\theta=1+1/s$, $\Psi(r)-r\sim\lambda r^{1-s}\to\infty$, again giving $\sigma\ne0$.
\end{proof}

Since $\Psi(r)\sim r^{s(\theta-1)}$ at infinity and complete Bernstein functions have at most linear growth, the upper bound $\theta\le1+1/s$ in \cref{prop:power-quotient} is optimal.

\begin{remark}\label{rem:integer-iteration}
    If $s=1/m$ and $\theta=m$ for an integer $m\ge2$, then
    \[
        \frac{r-\lambda^m}{r^{1/m}-\lambda}
        =\sum_{j=0}^{m-1}\lambda^j r^{(m-1-j)/m}.
    \]
    This gives the finite sum used in \cite[proof of Theorem~1.1]{KassmannSilvestre2014}.
\end{remark}

We now pass from these representations to solutions of local equations.

\begin{lemma}\label{lem:Bessel-potentials}
    Let $\Omega\subset\R^n$ be open, let $\mu\ge0$ be a nonzero tempered Radon measure supported in $\Omega^c$, and set $v_r=(-\Delta+r)^{-1}\mu$ for $r>0$. Let $\sigma\ge0$ satisfy $\int_0^\infty(1+r)^{-1}\sigma(\mathrm dr)<\infty$. If
    \[
        w(x)\coloneqq\int_0^\infty r v_r(x)\,\sigma(\mathrm dr)
        \in L^1_{\mathrm{loc}}(\Omega),
    \]
    then $w\in C^\infty_{\mathrm{loc}}(\Omega)$ and, for every $k\in\mathbb N_0$,
    \[
        \Delta^k w(x)=\int_0^\infty r^{k+1}v_r(x)\,\sigma(\mathrm dr)\ge0.
    \]
    All derivatives may be taken under the integral, and the inequalities are strict if $\sigma\ne0$.
\end{lemma}

\begin{proof}
    The Bessel kernel $K_r$ is strictly positive and decays exponentially away from the origin. Since $\mu$ is tempered, $v_r=K_r*\mu$ is smooth in $\Omega$. Indeed, $\int_{\R^n}(1+|y|)^{-N}\,\mu(\mathrm dy)<\infty$ for some $N$, and exponential decay makes every kernel derivative integrable against $\mu$ locally uniformly at points outside $\supp\mu$. The kernel is strictly positive and $\mu$ is non-zero, so $v_r>0$. The equation $(-\Delta+r)v_r=0$ there gives
    \begin{equation}\label{eq:Bessel-iterates-general}
        v_r>0,\qquad \Delta^k v_r=r^k v_r
        \quad\text{in }\Omega,\qquad k\in\mathbb N_0.
    \end{equation}
    For $U\Subset\Omega$ and a multi-index $\alpha$, the positive distance from $U$ to $\supp\mu$, together with the scaling $K_r(x)=r^{n/2-1}K_1(\sqrt r x)$ and exponential kernel estimates, gives
    \begin{equation}\label{eq:local-Bessel-estimates}
        \sup_U|\partial^\alpha v_r|\le Ce^{-c\sqrt r},\qquad r\ge1,
    \end{equation}
    where $C=C(U,\alpha,n,\mu)$ and $c=c(U,\mu)>0$; see \cite[Eqs.~(3.5), (3.7), and~(4.1)]{ArSm61}. The positive separation leaves an exponential factor in $\sqrt r$, while the remaining decay absorbs the polynomial weight needed to integrate against $\mu$.

    Near $r=0$, the assumption on $w$ supplies the integrability needed to differentiate.
    For $U\Subset V\Subset\Omega$, the inequality $\Delta v_r=rv_r\ge0$ gives local boundedness in terms of $\int_Vv_r$. Interior estimates for $\Delta-r$ then give, uniformly for $0<r\le1$,
    \begin{equation}\label{eq:bessel-aux-regularity}
        \sup_U|\partial^\alpha v_r|\le C\int_V v_r\,\mathrm dx.
    \end{equation}
    By Tonelli's theorem, $w\in L^1(V)$ implies $\int_0^1r\int_Vv_r\,\mathrm dx\,\sigma(\mathrm dr)<\infty$. Combining \eqref{eq:local-Bessel-estimates} and \eqref{eq:bessel-aux-regularity} with $r^{k+1}\le r$ for $0<r\le1$ and $r^{k+1}\exp(-c\sqrt r)\le C(1+r)^{-1}$ for $r\ge1$ gives
    \begin{equation}\label{eq:Bessel-sigma-domination}
        \int_0^\infty r^{k+1}\sup_U|\partial^\alpha v_r|\,\sigma(\mathrm dr)<\infty.
    \end{equation}
    Locally uniform domination of every derivative permits differentiation under the integral sign. Thus \eqref{eq:Bessel-iterates-general} gives all the stated identities. Since $r^{k+1}v_r>0$, the inequalities are strict when $\sigma\ne0$.
\end{proof}

\section{Superharmonicity in general sets}\label{sec:comparison}

	We now prove \cref{thm:power}. The proof is organized around positive
	solutions of a family of local equations.

\begin{proof}[Proof of \cref{thm:power}]
    By \cref{prop:general-defect}, $\lambda>0$ and $\mu=\lambda u-(-\Delta)^s u$ is a nonzero non-negative measure supported in $\Omega^c$. Set $v_t=(-\Delta+t)^{-1}\mu$. Interior regularity gives $u\in C^\infty_{\mathrm{loc}}(\Omega)$; see \cite[Corollary~7.20]{Grubb2009}. Positivity and simplicity follow from \cite[Propositions~7.1 and~3.4]{FranzinaLicheri2022}. Since $u\in L^2(\R^n)$ and is smooth in $\Omega$, localization also gives $(-\Delta)^\alpha u\in C^\infty_{\mathrm{loc}}(\Omega)$ for every $\alpha>0$. Indeed, one may localize by a cutoff; the local part is handled by smoothness of $u$, while the kernel acting on the $L^2$ remainder is smooth across the separated supports.

    We now express the differences in \cref{thm:power} as potentials of $\mu$. The estimates below justify this first in distributions; the local regularity will then give pointwise inequalities.
    For $\eta\in C_c^\infty(\Omega)$, the bounds
    \begin{align*}
        \|(-\Delta+t)^{-1}\eta\|_{H^s}&\le t^{-1}\|\eta\|_{H^s},\\
        \|(-\Delta)(-\Delta+t)^{-1}\eta\|_{H^s}
        &\le\min\{\|\eta\|_{H^s},t^{-1}\|(-\Delta)\eta\|_{H^s}\}
    \end{align*}
    justify the representations in \cref{prop:power-quotient} as $H^s$-valued integrals. In \eqref{eq:Psi-low-Stieltjes}, the first bound is integrable against $\sigma$; for \eqref{eq:CBF-general}, the second has the required $(1+t)^{-1}$ decay. We may therefore pair them with $\mu\in H^{-s}$ and interchange the pairing and the integral. The identity
    \[
        \lambda^\theta-r^{s\theta}=\Psi(r)(\lambda-r^s)
    \]
    gives, for $0<\theta<1$,
    \begin{equation}\label{eq:low-power-resolvent}
        \big\langle(\lambda^\theta-(-\Delta)^{s\theta})u,\eta\big\rangle
        =\langle\mu,\Psi(-\Delta)\eta\rangle
        =\int_0^\infty\int_\Omega v_t\eta\,\mathrm dx\,\sigma(\mathrm dt).
    \end{equation}
    Here symmetry of the Bessel kernel and Fubini's theorem move the resolvent from $\eta$ to $\mu$.
    For $1<\theta\le1+1/s$, we use \eqref{eq:CBF-general}. Since $\eta$ and $(-\Delta)\eta$ are supported in $\Omega$, their pairings with $\mu$ vanish. Using $(-\Delta)(-\Delta+t)^{-1}=\mathrm{Id}-t(-\Delta+t)^{-1}$, each resolvent term pairs with $\mu$ as $-t\langle\mu,(-\Delta+t)^{-1}\eta\rangle$. This minus sign cancels the sign of the reversed difference in \eqref{eq:iterated-subharmonicity}, and we obtain
    \begin{equation}\label{eq:high-power-resolvent}
        \big\langle(( -\Delta)^{s\theta}-\lambda^\theta)u,\eta\big\rangle
        =-\langle\mu,\Psi(-\Delta)\eta\rangle
        =\int_0^\infty t\int_\Omega v_t\eta\,\mathrm dx\,\sigma(\mathrm dt).
    \end{equation}
    Since $v_t\ge0$, testing \eqref{eq:low-power-resolvent} and \eqref{eq:high-power-resolvent} with a nonnegative cutoff equal to one on a compact subset of $\Omega$ bounds the potential integrals over that subset. Thus the corresponding potentials belong to $L^1_{\mathrm{loc}}(\Omega)$. Apply \cref{lem:Bessel-potentials} with $\sigma(\mathrm dt)/t$ in \eqref{eq:low-power-resolvent}, using $\int\sigma(\mathrm dt)/t<\infty$, and with $\sigma(\mathrm dt)$ in \eqref{eq:high-power-resolvent}. The resulting smooth potentials represent the two left-hand sides as distributions. Since those are smooth in $\Omega$, the identities hold pointwise. Since $\sigma\ne0$, every Laplacian iterate is strictly positive, proving \eqref{eq:iterated-low-comparison} and \eqref{eq:iterated-subharmonicity}.
\end{proof}

\begin{proof}[Proof of \cref{thm:superharmonicity}]
    Positivity, smoothness, and simplicity were established in the proof of \cref{thm:power}. Taking $\theta=1/s$ and $k=0$ in \eqref{eq:iterated-subharmonicity} gives $-\Delta u>\lambda^{1/s}u>0$. Dividing by $u$ and using the chain rule gives
    \[
        -\Delta\log u=\frac{-\Delta u}{u}+|\nabla\log u|^2
        >\lambda^{1/s}+|\nabla\log u|^2.
    \]
\end{proof}

\begin{proof}[Proof of \cref{cor:one-dimensional}]
		Since $-\Delta=-\mathrm d^2/\mathrm dx^2$ in dimension one, both assertions follow immediately from \cref{thm:superharmonicity}.
	\end{proof}
	
	Let us now prove \cref{thm:general-Bernstein}. The idea is very close to the proof of \cref{thm:power}:
    
\begin{proof}[Proof of \cref{thm:general-Bernstein}]
    The proof follows that of \cref{thm:power}, replacing its power quotient by the complete Bernstein function in \cref{prop:G-is-Bernstein}.
    Let $\mu=\lambda u-\psi(-\Delta)u$ be the nonzero exterior measure in \cref{prop:general-defect}, and set $\gamma=\psi^{-1}(\lambda)>0$ and $g=(-\Delta-\gamma)u$. By \cref{prop:G-is-Bernstein},
    \[
        G(t)=\frac{t-\gamma}{\psi(t)-\lambda}
    \]
    is complete Bernstein, with representing measure $\sigma$ in \eqref{eq:CBF-general}. The identity $G(t)(\lambda-\psi(t))=\gamma-t$ gives $g=-G(-\Delta)\mu$ by duality, where $\langle G(-\Delta)\mu,\eta\rangle\coloneqq\langle\mu,G(-\Delta)\eta\rangle$. This pairing is well-defined since $G(-\Delta)\eta\in H^\psi$ for $\eta\in C_c^\infty(\R^n)$: both $G$ and $\psi$ have at most linear growth, while $\widehat\eta$ is rapidly decreasing.

    The argument giving \eqref{eq:high-power-resolvent}, with $G$ and $H^\psi$ in place of $\Psi$ and $H^s$, yields
    \begin{equation}\label{eq:general-resolvent-formula}
        \langle g,\eta\rangle=
        \int_0^\infty r\int_\Omega v_r\eta\,\mathrm dx\,\sigma(\mathrm dr),
        \qquad v_r=(-\Delta+r)^{-1}\mu.
    \end{equation}
    Nonnegative cutoff functions show that $\int_0^\infty rv_r\,\sigma(\mathrm dr)\in L^1_{\mathrm{loc}}(\Omega)$. Thus \cref{lem:Bessel-potentials} gives $g\in C^\infty_{\mathrm{loc}}(\Omega)$ and
    \begin{equation}\label{eq:general-resolvent-hierarchy}
        \Delta^k g(x)=\int_0^\infty r^{k+1}v_r(x)\,\sigma(\mathrm dr)\ge0,
        \qquad k\in\mathbb N_0.
    \end{equation}
    Interior regularity for $-\Delta u=\gamma u+g$ now gives $u\in C^\infty_{\mathrm{loc}}(\Omega)$.

    If $u(x_0)=0$ at $x_0\in\Omega$, then $x_0$ is a local minimum and $\Delta u(x_0)\ge0$. Since $u\ge0$ is nonzero and $J>0$, we obtain
    \[
        \psi(-\Delta)u(x_0)=-b\Delta u(x_0)
        -\int_{\R^n}u(y)J(x_0-y)\,\mathrm dy<0,
    \]
    contradicting \eqref{eq:eigenproblem-bernstein}. Hence $u>0$ in $\Omega$. If $\psi$ is unbounded, \cref{prop:G-is-Bernstein} gives $\sigma\ne0$, so all inequalities in \eqref{eq:general-resolvent-hierarchy} are strict by \cref{lem:Bessel-potentials}.
\end{proof}

\begin{proof}[Proof of \cref{cor:eigenvalue-comparison}]
    By \cref{thm:superharmonicity}, $u>0$ is smooth in $\Omega$.
    For every $\eta\in C_c^\infty(\Omega)$, integration by parts gives
    \begin{equation}\label{eq:ground-state-representation}
        \int_\Omega|\nabla\eta|^2\,\mathrm dx
        -\int_\Omega\frac{-\Delta u}{u}\eta^2\,\mathrm dx
        =\int_\Omega u^2\left|\nabla\left(\frac\eta u\right)\right|^2
        \,\mathrm dx\ge0.
    \end{equation}
    Dividing by $\|\eta\|_2^2$ and taking the infimum over nonzero test functions yields
    $\inf_\Omega(-\Delta u/u)\le\mu_1(\Omega)$.

    For the strict comparison, we integrate the pointwise inequality against the square of a nonzero classical minimizer.
    Since $|\Omega|<\infty$, let $\phi_1\in H_0^1(\Omega)$ be an
    $L^2$-normalized minimizer for $\mu_1(\Omega)$. Approximate $\phi_1$
    in $H_0^1(\Omega)$ and almost everywhere by
    $\eta_j\in C_c^\infty(\Omega)$. Since
    $-\Delta u/u>\lambda_s(\Omega)^{1/s}>0$, Fatou's lemma and
    \cref{eq:ground-state-representation} give
    \[
        \lambda_s(\Omega)^{1/s}
        <\int_\Omega\frac{-\Delta u}{u}\phi_1^2\,\mathrm dx
        \le\lim_{j\to\infty}\int_\Omega|\nabla\eta_j|^2\,\mathrm dx
        =\mu_1(\Omega).
    \]
    The first inequality is strict since $\phi_1\not\equiv0$.
\end{proof}

\section{Log-concavity in a ball}\label{sec:log}

\begin{proof}[Proof of \cref{cor:ball-log-concavity}]
    By rotation invariance and simplicity, $u$ and
    $g=(-\Delta-\psi^{-1}(\lambda))u$ are radial. Indeed, rotating $u$ preserves the Rayleigh quotient, non-negativity, and its $L^2$ norm, so simplicity gives the same function. By
    \cref{thm:general-Bernstein}, $g,\Delta g\ge0$ and $\Delta u<0$.
    We use $u$ and $g$ also for their radial profiles, so $\Delta g=r^{1-n}(r^{n-1}g')'$ for $r>0$.
    Since $u'(0)=g'(0)=0$, integration of the radial equations gives, for $0<r<R$,
    \[
    \begin{aligned}
        r^{n-1}g'(r)&=\int_0^r t^{n-1}\Delta g(t)\,\mathrm dt\ge0,\\
        r^{n-1}u'(r)&=\int_0^r t^{n-1}\Delta u(t)\,\mathrm dt<0.
    \end{aligned}
    \]
    Therefore $h=g/u$ satisfies $h'=(g'u-gu')/u^2\ge0$.

    Set $q=-u'/u$. For $r>0$, the radial and tangential eigenvalues of $D^2\log u$ are $-q'$ and $-q/r$, respectively.
    Smoothness and radial symmetry give $u'(0)=0$ and $\Delta u(0)=nu''(0)$. Evaluating the equation for $u$ at the origin yields
    $q(0)=0$ and
    \begin{equation}\label{eq:q-origin}
        a\coloneqq q'(0)=\frac{\psi^{-1}(\lambda)+h(0)}n>0.
    \end{equation}
    Thus $D^2\log u(0)=-a\,\Id$, which is the value to compare with the Hessian away from the origin.
    It remains to show that $q/r>a$ and $q'>a$ on $(0,R)$.
    On $(0,R)$, $q>0$. Since $q'=-u''/u+q^2$, the equation for $u$ becomes
    \begin{equation}\label{eq:ball-riccati}
        q'=q^2-\frac{n-1}{r}q+\psi^{-1}(\lambda)+h.
    \end{equation}
    The restriction of a smooth radial function to any line through the origin is even. Taylor's formula therefore gives $q(r)=ar+O(r^3)$ as $r\downarrow0$.
    Since $h\ge h(0)$, substituting $na=\psi^{-1}(\lambda)+h(0)$ from \eqref{eq:q-origin} into \eqref{eq:ball-riccati} cancels the constant terms and gives
    \[
        \left[r^n\left(\frac qr-a\right)\right]'
        =r^{n-1}\bigl(q^2+h-h(0)\bigr)>0.
    \]
    Since $r^n(q/r-a)\to0$ as $r\downarrow0$, integration gives $q/r>a$. Differentiating
    \cref{eq:ball-riccati} and using the integrating factor $r^{n-1}u^2$, we then obtain
    \[
        \bigl[r^{n-1}u^2(q'-a)\bigr]'
        =r^{n-1}u^2\left[2aq+\frac{n-1}{r}
        \left(\frac qr-a\right)+h'\right]>0.
    \]
    Here the coefficient of $q'-a$ in the differentiated equation is $2q-(n-1)/r$. Since $(r^{n-1}u^2)'/(r^{n-1}u^2)=(n-1)/r-2q$, multiplication by the integrating factor cancels this term.
    Since $q'(r)\to a$ as $r\downarrow0$, another integration gives $q'>a$.
    Thus, for $x\ne0$,
    \[
        D^2\log u(x)<D^2\log u(0)
        =-a\Id=\frac{\Delta u(0)}{nu(0)}\Id
        \le-\frac{\psi^{-1}(\lambda)}n\Id.
    \]
    If $\psi$ is unbounded, \cref{thm:general-Bernstein} gives $g(0)>0$,
    making the last inequality strict.
\end{proof}

    \section{Failure of the fractional Brascamp--Lieb property}
\label{sec:counter-example}

In this section we prove \cref{thm:counter-example}. Fix $n\ge2$ and
$s\in(0,1)$. Throughout the section, $c,C>0$ depend only on $n$ and $s$.
Recall that
\[
    \Omega_\varepsilon
    =\left\{(t,z)\in\R\times\R^{n-1}:
    t^2+\frac{|z|^2}{\varepsilon^2}<1\right\},
    \qquad 0<\varepsilon<1.
\]
Let $u_\varepsilon$ be its positive, $L^2$-normalized fractional ground
state, extended by zero to $\R^n$, and write
$\lambda_\varepsilon=\lambda_s(\Omega_\varepsilon)$ and
$\Lambda=\lambda_s(B_1^{n-1})$.
Boundedness and boundary regularity give
$u_\varepsilon\in L^\infty(\R^n)\cap C(\R^n)$ for each fixed
$\varepsilon>0$; see \cite{FernandezRealRosOton2024}. Set
\begin{equation}\label{eq:counter-rho}
    \rho_\varepsilon(t)=\int_{\R^{n-1}}u_\varepsilon(t,z)^2\,\mathrm dz,
    \qquad
    M_\varepsilon=\int_\R t^2\rho_\varepsilon(t)\,\mathrm dt.
\end{equation}
Thus $\rho_\varepsilon$ describes the distribution of the $L^2$ mass along
the long axis of the ellipsoid, with $\int_\R\rho_\varepsilon=1$ and
$M_\varepsilon>0$.
Uniqueness of the positive normalized ground state and invariance under
$t\mapsto-t$ make $\rho_\varepsilon$ even. It is also continuous by
dominated convergence, since $u_\varepsilon$ is continuous and bounded
and all its transverse slices are supported in $B_\varepsilon^{n-1}$.

The proof compares two estimates for $\rho_\varepsilon(1/2)$.
The transverse sections are widest at $t=0$, where their first eigenvalue
is smallest. This forces the $L^2$ mass to concentrate near $t=0$, as
shown in \cref{lem:counter-localization}. Under log-concavity, this
concentration gives an exponential upper bound in
\cref{lem:counter-exponential}. On the other hand, the nonlocal equation
couples the values at $t=1/2$ to the mass in the rest of the ellipsoid;
\cref{lem:counter-polynomial} gives a polynomial lower bound, leading to
a contradiction.

\begin{lemma}\label{lem:counter-localization}
    For all sufficiently small $\varepsilon$, there is $C$ independent of $\varepsilon$ such that
    \begin{equation}\label{eq:counter-localization}
        M_\varepsilon\le C\varepsilon^s,
        \qquad
        \lambda_\varepsilon\le C\varepsilon^{-2s}.
    \end{equation}
\end{lemma}

\begin{proof}
    We first bound the energy from below by the transverse energies of
    the slices. For $v\in C_c^\infty(\Omega_\varepsilon)$, write the
    Fourier variable as $(\tau,\zeta)\in\R\times\R^{n-1}$.
    Plancherel's (in the variable $t$) and
    $(\tau^2+|\zeta|^2)^s\ge|\zeta|^{2s}$ give
    \[
        \langle v,v\rangle_s
        \ge\int_{-1}^1\frac1{(2\pi)^{n-1}}
        \int_{\R^{n-1}}|\zeta|^{2s}
        |\mathcal F_zv(t,\zeta)|^2\,\mathrm d\zeta\,\mathrm dt.
    \]
    The slice $v(t,\cdot)$ is supported in
    $B_{\varepsilon\sqrt{1-t^2}}^{n-1}$, whose first eigenvalue is
    $\Lambda\varepsilon^{-2s}(1-t^2)^{-s}$ by scaling. Applying its
    variational characterization (i.e., using the Rayleigh quotient) to each slice, we obtain
    \[
        \langle v,v\rangle_s
        \ge\Lambda\varepsilon^{-2s}\int_{-1}^1(1-t^2)^{-s}
        \int_{\R^{n-1}}v(t,z)^2\,\mathrm dz\,\mathrm dt.
    \]
    Approximate $u_\varepsilon$ in $H^s(\R^n)$ by functions in
    $C_c^\infty(\Omega_\varepsilon)$. After passing to a subsequence, the
    slices converge in $L^2_z$ for almost every $t$. Fatou's lemma and
    $(1-t^2)^{-s}\ge1+st^2$ therefore yield
    \begin{equation}\label{eq:counter-localization-lower}
        \lambda_\varepsilon
        \ge\Lambda\varepsilon^{-2s}
        \int_{-1}^1(1-t^2)^{-s}\rho_\varepsilon(t)\,\mathrm dt
        \ge\Lambda\varepsilon^{-2s}(1+sM_\varepsilon).
    \end{equation}
    Thus mass away from $t=0$ increases the energy above
    $\Lambda\varepsilon^{-2s}$.

    For the opposite bound, we place a product test function near the
    widest section. Assume $0<\varepsilon<1/8$, fix once and for all
    $f\in C_c^\infty((-1,1))$ with $\|f\|_2=1$, and set
    $\delta=\sqrt\varepsilon$ and $R=\varepsilon\sqrt{1-4\varepsilon}$.
    If $\phi_R$ is the normalized ground state in $B_R^{n-1}$, define
    $w(t,z)=\delta^{-1/2}f(t/\delta)\phi_R(z)$, so that $\|w\|_2=1$.
    Its support is compactly contained in $\Omega_\varepsilon$, since
    \[
        t^2+\frac{|z|^2}{\varepsilon^2}
        \le\delta^2+\frac{R^2}{\varepsilon^2}
        =1-3\varepsilon<1
        \qquad\text{on }\supp w.
    \]
    Moreover, $(\tau^2+|\zeta|^2)^s\le|\tau|^{2s}+|\zeta|^{2s}$ shows
    that $w\in H^s(\R^n)$ and bounds its energy by
    $C\delta^{-2s}+\Lambda R^{-2s}$, the sum of the longitudinal and
    transverse contributions. Interior mollification then gives
    $w\in H_0^s(\Omega_\varepsilon)$, so
    \begin{equation}\label{eq:counter-localization-upper}
        \lambda_\varepsilon\le\langle w,w\rangle_s
        \le C\varepsilon^{-s}
        +\Lambda\varepsilon^{-2s}(1-4\varepsilon)^{-s}.
    \end{equation}
    Our choice of $\delta$ makes the longitudinal contribution lower
    order than $\Lambda\varepsilon^{-2s}$, while the transverse radius
    differs from $\varepsilon$ by a relative error of order
    $\varepsilon$. In particular, \eqref{eq:counter-localization-upper}
    gives the bound on $\lambda_\varepsilon$ in
    \eqref{eq:counter-localization}. Comparing
    \cref{eq:counter-localization-lower,eq:counter-localization-upper}
    and using $(1-4\varepsilon)^{-s}\le1+C\varepsilon$ gives
    $M_\varepsilon\le C(\varepsilon+\varepsilon^s)\le C\varepsilon^s$.
\end{proof}

We next show how log-concavity would turn the concentration in
\cref{lem:counter-localization} into an exponential bound away from the
middle of the ellipsoid.

\begin{lemma}\label{lem:counter-exponential}
    If $u_\varepsilon$ is log-concave then, for all sufficiently small $\varepsilon$, there are $c$ and $C$ independent of $\varepsilon$ such that
    \begin{equation}\label{eq:counter-exponential}
        \rho_\varepsilon(1/2)
        \le C\exp\!\left(-c\varepsilon^{-s/2}\right).
    \end{equation}
\end{lemma}

\begin{proof}
    Since $\Omega_\varepsilon$ is convex, the zero extension of
    $u_\varepsilon^2$ is log-concave on $\R^n$. Pr\'ekopa's theorem
    \cite[Corollary~3.5]{BrascampLieb1976} shows that its marginal $\rho_\varepsilon$
    is log-concave. The same theorem applies to
    $\rho_\varepsilon(t+r)\mathbf1_{\{r\ge0\}}$, so the tail
    \[
        F(t)=\int_t^\infty\rho_\varepsilon(r)\,\mathrm dr
        =\int_0^\infty\rho_\varepsilon(t+r)\,\mathrm dr
    \]
    is log-concave as well. Working with $F$ allows us to use the second
    moment directly: evenness gives $F(0)=1/2$, and Chebyshev's inequality
    gives
    \[
        F(2\sqrt{M_\varepsilon})
        =\frac12\int_{|r|\ge2\sqrt{M_\varepsilon}}
        \rho_\varepsilon(r)\,\mathrm dr
        \le\frac1{8M_\varepsilon}
        \int_\R r^2\rho_\varepsilon(r)\,\mathrm dr
        =\frac18.
    \]
    By \cref{lem:counter-localization}, we may take $\varepsilon$ small
    enough that $2\sqrt{M_\varepsilon}<1/4$.
    Concavity of $\log F$ means that its secant slopes from $0$ are
    nonincreasing. The drop from $1/2$ to at most $1/8$ over a distance
    $2\sqrt{M_\varepsilon}$ therefore gives
    \begin{equation}\label{eq:counter-tail}
        F(t)\le\frac12
        \exp\!\left(-\frac{(\log2)t}{\sqrt{M_\varepsilon}}\right),
        \qquad t\ge2\sqrt{M_\varepsilon}.
    \end{equation}

    To recover a pointwise bound, note that evenness and log-concavity
    make $\rho_\varepsilon$ nonincreasing on $[0,\infty)$. Hence
    \[
        \frac14\rho_\varepsilon(1/2)
        \le\int_{1/4}^{1/2}\rho_\varepsilon(t)\,\mathrm dt
        \le F(1/4).
    \]
    Applying \eqref{eq:counter-tail} at $t=1/4$ and then
    \cref{lem:counter-localization}, we obtain
    $\rho_\varepsilon(1/2)\le2\exp(-c/\sqrt{M_\varepsilon})
    \le C\exp(-c\varepsilon^{-s/2})$, proving
    \eqref{eq:counter-exponential}.
\end{proof}

The nonlocal equation gives a lower bound of a different order.
The argument in \cref{lem:counter-polynomial} first controls the total
$L^1$ mass, then uses the Poisson kernel of a small ball to transfer the
mass outside the ball to its center.

\begin{lemma}\label{lem:counter-polynomial}
    For all sufficiently small $\varepsilon$, there is $c$ independent of $\varepsilon$ such that
    \begin{equation}\label{eq:counter-polynomial}
        \rho_\varepsilon(1/2)\ge c\varepsilon^{4s+2n-1}.
    \end{equation}
\end{lemma}

\begin{proof}
    We first obtain the $L^1$ lower bound from the normalization and the
    energy estimate in \cref{lem:counter-localization}. Since $n>2s$, set
    $p=2n/(n-2s)$. For $v\in H^s(\R^n)\cap L^1(\R^n)$, H\"older's
    inequality and the fractional Sobolev inequality give
    \[
        \|v\|_2\le\|v\|_1^{2s/(n+2s)}\|v\|_p^{n/(n+2s)}
        \le C\|v\|_1^{2s/(n+2s)}
        \langle v,v\rangle_s^{n/(2(n+2s))}.
    \]
    Raising to the power $2+4s/n$ yields the fractional Nash inequality
    \begin{equation}\label{eq:counter-nash}
        \|v\|_2^{2+4s/n}\le C\langle v,v\rangle_s\|v\|_1^{4s/n}.
    \end{equation}
    Applying \eqref{eq:counter-nash} to $u_\varepsilon$, using
    $\|u_\varepsilon\|_2=1$,
    $\langle u_\varepsilon,u_\varepsilon\rangle_s=\lambda_\varepsilon$,
    and \cref{lem:counter-localization}, gives
    \begin{equation}\label{eq:counter-L1}
        \int_{\Omega_\varepsilon}u_\varepsilon\,\mathrm dx
        \ge c\lambda_\varepsilon^{-n/(4s)}
        \ge c\varepsilon^{n/2}.
    \end{equation}

    Fix $x=(1/2,z)$ with $|z|\le\varepsilon/4$, and set
    $r=\varepsilon/16$ and $B=B_r(x)$.
    We will use the mass outside $B$ in the Poisson comparison. First,
    $B\Subset\Omega_\varepsilon\cap\{t>1/4\}$: every $y=(y_1,y')\in B$
    satisfies $|y_1|\le9/16$ and $|y'|/\varepsilon\le5/16$, whence
    $y_1^2+|y'|^2/\varepsilon^2\le106/256<1$.
    Since $t>1/4$ on $B$, \cref{eq:counter-rho,lem:counter-localization}
    give
    \[
        \int_Bu_\varepsilon^2\,\mathrm dx
        \le\int_{|t|>1/4}\rho_\varepsilon(t)\,\mathrm dt
        \le16M_\varepsilon\le C\varepsilon^s.
    \]
    Cauchy--Schwarz and $|B|=c_nr^n$ therefore yield
    \[
        \int_Bu_\varepsilon\,\mathrm dx
        \le |B|^{1/2}\|u_\varepsilon\|_{L^2(B)}
        \le C\varepsilon^{(n+s)/2}.
    \]
    This is $o(\varepsilon^{n/2})$ because $s>0$. Thus the ball contains
    only a small fraction of the $L^1$ mass in \eqref{eq:counter-L1},
    and, for small $\varepsilon$,
    \begin{equation}\label{eq:counter-L1-outside}
        \int_{\Omega_\varepsilon\setminus B}u_\varepsilon\,\mathrm dx
        \ge c\varepsilon^{n/2}.
    \end{equation}

    Let $h$ be the $s$-harmonic extension to $B$ of
    $u_\varepsilon|_{B^c}$. Then
    \[
        (-\Delta)^s(u_\varepsilon-h)
        =\lambda_\varepsilon u_\varepsilon\ge0\quad\text{in }B,
        \qquad u_\varepsilon-h=0\quad\text{in }B^c,
    \]
    so the maximum principle gives $u_\varepsilon\ge h$ in $B$.
    At the center $x$, the Poisson representation reads
    \[
        h(x)=\int_{\Omega_\varepsilon\setminus B}P_B(x,y)
        u_\varepsilon(y)\,\mathrm dy,
        \qquad
        P_B(x,y)=a_{n,s}\frac{r^{2s}}
        {(|x-y|^2-r^2)^s|x-y|^n},
    \]
    where $a_{n,s}>0$; see \cite[Theorem~2.10]{Buc16}.
    Since $\Omega_\varepsilon\subset B_1^n$, we have $|x-y|\le2$.
    The denominator is therefore bounded above independently of
    $\varepsilon$, and $r=\varepsilon/16$ gives
    $P_B(x,y)\ge c\varepsilon^{2s}$ for almost every
    $y\in\Omega_\varepsilon\setminus B$. Combining this with
    \eqref{eq:counter-L1-outside} and $u_\varepsilon\ge h$, we obtain
    \begin{equation}\label{eq:counter-pointwise}
        u_\varepsilon(1/2,z)\ge c\varepsilon^{2s+n/2}
        \qquad\text{for }|z|\le\varepsilon/4.
    \end{equation}
    Finally, the set $\{|z|\le\varepsilon/4\}$ has measure
    $c\varepsilon^{n-1}$. Squaring \eqref{eq:counter-pointwise} and
    integrating over this set gives
    \[
        \rho_\varepsilon(1/2)
        \ge c\varepsilon^{n-1}\varepsilon^{2(2s+n/2)}
        =c\varepsilon^{4s+2n-1},
    \]
    as claimed.
\end{proof}

We can finally show \Cref{thm:counter-example}:

\begin{proof}[Proof of \cref{thm:counter-example}]
    If $u_\varepsilon$ were log-concave, then
    \cref{lem:counter-exponential,lem:counter-polynomial} would give
    \[
        c\varepsilon^{4s+2n-1}
        \le C\exp\!\left(-c\varepsilon^{-s/2}\right),
    \]
    which is a contradiction for $\varepsilon> 0$ sufficiently small. 
\end{proof}

\section*{Acknowledgements}
	
	N.~De Nitti is a member of the Gruppo Nazionale per l'Analisi Matematica, la Probabilit\`a e le loro Applicazioni (GNAMPA) of the Istituto Nazionale di Alta Matematica (INdAM) and has received support from the INdAM--GNAMPA Project 2026 \textit{Modelli Non-locali in Fluidodinamica, Traffico ed Elasticit\`a} (CUP:~E53C25002010001). He also acknowledges the support of the Italian Ministry of University and Research through the ``Dipartimenti di Eccellenza'' Programme, Law 232/2016 (CUP~D94I18000260001).

	X.~Fern\'andez-Real is supported by the Swiss State Secretariat for Education, Research and Innovation (SERI) under contract number MB22.00034 through the project TENSE, by the Swiss National Science Foundation under grant PZ00P2\_208930, and by the AEI project PID2024-156429NB-I00.
	
	F.~Grube was financially supported by the German Research Foundation (DFG, project number 541771122).
	
	The authors are grateful to Moritz Kassmann for helpful discussions.

    \subsubsection*{On the development of the article}
    The two groups of authors De Nitti \& Fern\'andez-Real and Grube \& Prasad worked simultaneously on the problem of superharmonicity of the fractional ground state. We decided to merge our approaches and results in a single joint work.
	
	\subsubsection*{AI disclosure}
    
	Large language models were used during the development of the original preprints and in producing an initial merged draft of the two approaches (using the environment developed for \cite{FS26}). Specifically, LLMs helped in identifying $G$, see \eqref{eq:factorization-formal}, as a complete Bernstein function. 

    Based on numerical experiments and upon suggesting that a sufficiently thin ellipsoid could break $\log$-concavity, the OpenAI model GPT-5.6 Sol constructed the counterexample for the general fractional Brascamp-Lieb theorem. \Cref{sec:counter-example} is based on this counterexample.
    
    The authors take full responsibility for checking every mathematical statement and rewriting the proof in the final version.
	
	\printbibliography
	
	\vfill
	
\end{document}